\documentclass[12 pt]{article}
\usepackage{fullpage}

\usepackage{amsfonts}
\usepackage{graphicx}
\usepackage{epstopdf}
\usepackage{standalone}
\usepackage{algcompatible}

\ifpdf
  \DeclareGraphicsExtensions{.eps,.pdf,.png,.jpg}
\else
  \DeclareGraphicsExtensions{.eps}
\fi

\usepackage{amsopn}
\usepackage[natbib=true,doi=false,isbn=false,url=false,
hyperref=true, giveninits=true, backend=biber, maxnames = 3]{biblatex}
\usepackage{epsfig} %
\usepackage{amsopn}
\usepackage{amsthm}
\usepackage{amsmath} %
\usepackage{amssymb}  %
\usepackage{mathptmx} %
\usepackage{bm}
\DeclareMathAlphabet{\mathcal}{OMS}{cmsy}{m}{n}
\usepackage[english]{babel}
\usepackage{url}
\usepackage{algorithm}
\usepackage{color}
\usepackage{pgfplots}
\usepackage{siunitx}
\usepackage{tikz}
\usepackage{tkz-euclide}
\usepackage{chngcntr}
\usepackage{verbatim}
\usepackage{enumerate}
\usepackage[normalem]{ulem}

\definecolor{ao(english)}{rgb}{0.0, 0.5, 0.0}
\usepackage[colorlinks, citecolor = {ao(english)}, linkcolor = {ao(english)}]{hyperref} 
\usepackage[nameinlink]{cleveref}
\usepackage{aliascnt}
\usetikzlibrary{calc}
\usepackage{subcaption}
\usepackage{tikz}
\usepackage{pgfplots}
\usetikzlibrary{arrows,shapes,trees,calc,positioning,patterns,decorations.pathmorphing,decorations.markings}
\usetikzlibrary{matrix}
\usepgfplotslibrary{groupplots}
\pgfplotsset{compat=newest}
\usepgfplotslibrary{patchplots}

\usetikzlibrary{external}
\crefname{figure}{Fig.}{Fig.}

\newtheorem{thm}{Theorem}[section]
\crefname{thm}{Theorem}{Theorems}

\newtheorem{prop}{Proposition}[section]
\crefname{prop}{Proposition}{Propositions}

\newtheorem{lem}{Lemma}[section]
\crefname{lem}{Lemma}{Lemmas}

\newtheorem{cor}{Corollary}[section]
\crefname{cor}{Corollary}{Corollaries}

\newtheorem{rem}{Remark}[section]
\crefname{rem}{Remark}{Remark}

\crefname{ass}{Assumption}{Assumption}

\crefname{conj}{Conjecture}{Conjectures}

\newtheorem{defn}{Definition}[section]
\crefname{defn}{Definition}{Definitions}

\crefname{prob}{Problem}{Problems}
\crefname{algorithm}{Algorithm}{Algorithms}
\crefname{paper}{Paper}{Papers}
\crefname{figure}{Figure}{Figures}
\crefname{section}{Section}{Sections}
\Crefname{section}{Section}{Sections}

\usepackage{dsfont}
\let\mathbb=\mathds

\newcommand{\Rmn}{\mathbb{R}^{n \times m}}

\newcommand{\Rnn}{\mathbb{R}^{n \times n}}

\newcommand{\cone}{\textnormal{cone}}
\newcommand{\conv}{\textnormal{conv}}
\newcommand{\rk}{\textnormal{rank}}

\newcommand{\diag}{\textnormal{diag}}

\newcommand{\cl}{\textnormal{cl}}

\newcommand{\transp}{\mathsf{T}}
\newcommand{\vari}[1]{\text{S}(#1)}
\newcommand{\varict}[1]{\text{S}_{\text{\tiny CT}}(#1)}

\newcommand{\tp}[1]{{#1\text{-positive}}}

\newcommand{\stp}[1]{\text{strictly }#1\text{-positive}}

\newcommand{\Hank}[1]{\mathcal{H}_{#1}}
\newcommand{\Con}[1]{{\mathcal{C}^{#1}}}
\newcommand{\Obs}[1]{{\mathcal{O}^{#1}}}
\newcommand{\Conct}[1]{{\mathcal{C}_{\text{\tiny CT}}^{#1}}}

\newcommand{\Obsct}[1]{{\mathcal{O}_{\text{\tiny CT}}^{#1}}}

\newcommand{\im}[1]{\textnormal{im}({#1})}

\newcommand{\vardim}[1]{${#1}$-variation diminishing}
\newcommand{\vardimord}[1]{order-preserving ${#1}$-variation diminishing}
\newcommand{\ovd}[1]{\text{OVD}_{#1}}

\newcommand{\vd}[1]{\text{VD}_{#1}}

\newcommand{\lint}{\ell_1}
\newcommand{\linf}{\ell_\infty}

\newcommand{\compound}[2]{{#1}_{[#2]}}

\colorlet{FigColor1}{blue}
\colorlet{FigColor2}{red}
\colorlet{FigColor3}{ao(english)}
\colorlet{FigColor4}{orange}
\pgfplotsset{every axis plot/.append style={line width=1.5pt}}

\crefformat{equation}{\textup{#2(#1)#3}}
\crefrangeformat{equation}{\textup{#3(#1)#4--#5(#2)#6}}
\crefmultiformat{equation}{\textup{#2(#1)#3}}{ and \textup{#2(#1)#3}}
{, \textup{#2(#1)#3}}{, and \textup{#2(#1)#3}}
\crefrangemultiformat{equation}{\textup{#3(#1)#4--#5(#2)#6}}%
{ and \textup{#3(#1)#4--#5(#2)#6}}{, \textup{#3(#1)#4--#5(#2)#6}}{, and \textup{#3(#1)#4--#5(#2)#6}}

\Crefformat{equation}{#2Equation~\textup{(#1)}#3}
\Crefrangeformat{equation}{Equations~\textup{#3(#1)#4--#5(#2)#6}}
\Crefmultiformat{equation}{Equations~\textup{#2(#1)#3}}{ and \textup{#2(#1)#3}}
{, \textup{#2(#1)#3}}{, and \textup{#2(#1)#3}}
\Crefrangemultiformat{equation}{Equations~\textup{#3(#1)#4--#5(#2)#6}}%
{ and \textup{#3(#1)#4--#5(#2)#6}}{, \textup{#3(#1)#4--#5(#2)#6}}{, and \textup{#3(#1)#4--#5(#2)#6}}

\crefdefaultlabelformat{#2\textup{#1}#3}

\title{Internally Hankel {$k$}-positive systems \thanks{This work received support by grants from ONR and NSF as well as under the Advanced ERC Grant Agreement Switchlet n.670645 and by DGAPA-UNAM under the grant PAPIIT RA105518. }}

\author{Christian Grussler \thanks{Department of Electrical Engineering and Computer Sciences, UC Berkeley, Berkeley, CA ({christian.grussler@berkeley.edu})}.
\and Thiago B. Burghi\thanks{Department of Engineering, University of Cambridge, Cambridge, UK ({tbb29@cam.ac.uk}).}
\and Somayeh Sojoudi  \thanks{Department of Electrical Engineering and Computer Sciences, UC Berkeley, Berkeley, CA ({sojoudi@berkeley.edu})} }

\begin{document}
\maketitle

	\begin{abstract}	
	    The classes of externally Hankel $k$-positive LTI systems and autonomous $k$-positive systems 
	    have recently been defined, and their properties and applications began to be explored
	    using the framework of total positivity and variation diminishing operators.
	    In this work, these two system classes are subsumed under a new class of internally 
	    Hankel $k$-positive systems, which we define as state-space LTI systems with 
	    $k$-positive controllability and observability operators. We show that internal Hankel 
	    $k$-positivity is a natural extension of the celebrated property of internal positivity ($k=1$), and we derive tractable conditions for verifying  the cases $k > 1$ in the form of internal positivity of the first $k$ 
	    compound systems. As these conditions define a new positive realization problem, 
	    we also discuss geometric conditions for when a minimal internally Hankel $k$-positive realization exists. Finally, we use our results to establish a new framework for bounding 
	    the number of over- and undershoots in the step response of general LTI systems.

\end{abstract}

\section{Introduction}

Externally positive linear time-invariant (LTI) systems 
\begin{equation}\label{eq:SISO_d}
    \begin{aligned}
 	    x(t+1) &= A x(t) + b u(t) \\
        y(t) & = cx(t),
    \end{aligned}
\end{equation}
mapping nonnegative inputs $u(t)$ to nonnegative outputs $y(t)$ have been
recognized as an important system class at least since the exposition by Luenberger \cite{Luenberger1979},
but many of their favourable properties have only recently been exploited \cite{rantzer2015scalable,tanaka2011bounded,Farina2000,son1996robust}.
Particular emphasis has been given to the subclass of internally positive 
systems, that is, externally positive systems such that $x(t)$ remains in 
the nonnegative orthant for nonnegative $u(t)$. As such systems are
characterized by nonnegative system matrices $A$, $b$ and $c$, they can be 
studied with finite-dimensional nonnegative matrix analysis 
\cite{berman1989nonnegative},
an advantage that motivated the search for conditions under 
which an externally positive system admits an internally positive realization
\cite{ohta1984reachability,anderson1996nonnegative,benvenuti2004tutorial}.

At the same time, externally positive systems have been studied for over a century from the viewpoint of fields such as statistics and interpolation theory, leading to the theory of total positivity 
\cite{karlin1968total}. 
Central to this theory is the study of \emph{variation-diminishing} convolution operators
\begin{equation}
    \label{eq:convolution}
    y(t) = \sum_{\tau=-\infty}^{\infty} g(t-\tau) u(\tau),
\end{equation}
with nonnegative kernels $g$, that bound
the \emph{variation} (number of sign changes) of $y(t)$
by the variation of $u(t)$. More generally, a linear mapping $u \mapsto Gu$ is called \emph{$k$-variation diminishing} ($\vd{k}$) if it maps an input $u$ with at most $k$ sign changes to an output
$Gu$ whose number of sign changes do not exceed those of $u$; if the order in
which sign changes occur is preserved whenever $u$ and $Gu$ share the same
number of sign variations, the $\vd{k}$ property is said to be
\emph{order-preserving} ($\ovd{k}$). A core result of total positivity is an
algebraic characterization: an operator is $\ovd{k}$ if and only its matrix
representation is $\tp{k}$, that is, all 
the minors of order up to $k$ in that matrix are nonnegative 
\cite{grussler2020variation,karlin1968total}; \emph{total} positivity refers to 
the case when this is true for all $k$. Under this framework, externally positive 
LTI systems are associated with $\ovd{0}$ Hankel and Toeplitz operators, while for
internally positive systems $x \mapsto Ax$ as well as the controllablity and 
observability operators are $\ovd{0}$.

Despite the link between positive systems and variation-diminishing operators,
it was not until very recently that $\ovd{k-1}$ and $k$-positivity have been studied as 
properties of LTI systems when $k > 1$. 
New results, with applications in non-linear systems analysis and model order reduction, 
have so far focused on two distinct cases: the \textit{external case} 
\cite{grussler2019tractable,grussler2020variation,grussler2020balanced}, where $\ovd{k}$ 
is considered as a property of system convolution operators; and related \textit{autonomous 
cases} \cite{margaliot2018revisiting,alseidi2021discrete}, which concern $k$-positivity of 
the state-space matrix $A$ with $b = 0$. A main result of \cite{grussler2020variation} 
characterizes \emph{(externally) Hankel $k$-positive systems}, i.e., systems with 
$\ovd{k-1}$ Hankel operators, in terms of the external positivity of all so-called 
\emph{$j$-compound systems}, $1 \leq j \leq k$, whose impulse responses are given by 
consecutive $j$-minors of the Hankel operator's matrix representation.

In this paper, we develop a realization theory of Hankel $k$-positivity based on the notion of \emph{internally Hankel 
$k$-positive systems}, which we define as
state-space systems where the controllability and observability operators 
as well as $A$ are $\ovd{k-1}$.
Not only does this theory enable the study of variation-diminishing systems through finite-dimensional 
analysis, but it also establishes an important first bridge between the aforementioned 
autonomous and external notions. Our main result is a finite-dimensional, tractable condition for the verification of 
the $\ovd{k-1}$ property of the controllability and observability operators.
Using this result, internal Hankel $k$-positivity can be completely characterized in terms 
of the existence of a realization that renders all $j$-compound systems internally positive, 
$1 \leq j \leq k$. 
We then use these insights to discuss geometric conditions for the existence of minimal internally Hankel $\tp{k}$ realizations, as previously done for the special case $k=1$ in \cite{ohta1984reachability}. In particular, it is easy to verify then that all relaxation systems \cite{willems1976realization} ($k=\infty$) have 
a minimal internally Hankel totally positive realization. 

As a practical application, we
show how our results can be used to obtain
upper bounds on the number of over- and undershoots in the step response
of an LTI system. This is a classical control problem that lies at the 
heart of the rise-time-settling-time trade-off 
\cite{aastrom2010feedback}, and for which several lower bounds
\cite{damm2014zero,swaroop1996some,elkhoury1993influence,elkhoury1993discrete}, 
but few upper bounds \cite{elkhoury1993influence,elkhoury1993discrete} have
been found. Our approach can be seen as a direct generalization of 
\cite{elkhoury1993influence,elkhoury1993discrete}.  
Non-linear extensions of this problem are of interest
both in control \cite{karlsson2020feedback} and online learning,
in the form of (static) regret \cite{orabona2019modern}; we thus
envision our work as the basis for possible interdisciplinary 
applications. Other possible contributions resulting from 
non-linear extension are discussed in \cite{grussler2020variation}.

The paper is organized as follows. In the preliminaries, we recap total positivity theory and externally Hankel $k$-positive systems. Then, we introduce the concept of internal Hankel $k$-positivity and present our main results on its characterization. Subsequently, extensions to continuous-time and applications to the determination of impulse response zero-crossings are discussed. We conclude with illustrative examples and a summary of open problems.

\section{Preliminaries}
	\label{sec:prelim}

This work lies at the interface between positive control systems and total
positivity theory. Alongside some standard notation, this section briefly introduces key concepts and results from these fields, including recent results on externally $k$-positive LTI systems, which are crucial to the motivation of our main results.

\subsection{Notations}
We write $\mathds{Z}$ for the set of integers and $\mathds{R}$ for the set of
reals, with  $\mathds{Z}_{\ge 0}$ and $\mathds{R}_{\ge 0}$ standing for the respective
subsets of nonnegative elements; the corresponding notation with strict inequality is also used for positive elements. 
The set of real sequences with indices in $\mathbb{Z}$ is denoted by 
$\mathbb{R}^{\mathbb{Z}}$.  
For matrices $X = (x_{ij}) \in \Rmn$, we say that $X$ is
\emph{nonnegative},
$X \geq 0$ 
or $X \in \Rmn_{\geq 0}$
if all elements $x_{ij} \in \mathbb{R}_{\geq 0}$
; again, we use the corresponding notations in case of \emph{positivity}. 
The notations also apply to sequences $x = (x_i) \in \mathbb{R}^{\mathbb{Z}}$.
Submatrices of
$X \in \Rmn$ are denoted by $X_{I,J} := (x_{ij})_{i \in I, j \in J}$, 
where $I {\subseteq} \{1,\dots,n\}$ and 
$J {\subseteq} \{1,\dots,m\}$. If $I$ and $J$
have cardinality $r$, then $\det(X_{I,J})$ is referred to as an
\emph{$r$-minor}, and as a \emph{consecutive $r$-minor} if $I$ and $J$ 
are intervals. 
For $X \in \Rnn$, $\sigma(X) = \{\lambda_1(X),\dots,\lambda_n(X)\}$
denotes the \emph{spectrum} of $X$, where the eigenvalues are ordered
by descending absolute value, i.e., $\lambda_1(X)$ is the eigenvalue
with the largest magnitude, counting multiplicity. In case that the
magnitude of two eigenvalues coincides, we sub-sort them by decreasing
real part. If there exists a permutation matrix $P=[P_1,P_2]$ so that
$P_2^TXP_1=0$, then $X$ is called \emph{reducible} and otherwise
\emph{irreducible}. Further, $X$ is said to be \emph{positive
semidefinite}, $X \succeq 0$, if $X = X^\transp$ and $\sigma(X) 
\subset \mathbb{R}_{\geq 0}$. Further, we use $I_n$ to denote the
identity matrix in $\Rnn$. For $\mathcal{S} \subset \mathds{R}^n$, we denote its \emph{closure}, \emph{convex hull} and $\emph{convex conic hull}$ by $\cl(\mathcal{S})$, $\conv(\mathcal{S})$ and $\cone(\mathcal{S})$, respectively. $\mathcal{S}$ is a \emph{polyhedral cone}, if there exists $k \in \mathds{Z}_{>0}$ and $P \in \mathds{R}^{n \times k}$ such that $\mathcal{S} = \{Px: x \in \mathds{R}^{k}_{\geq 0} \} =: \cone(P)$. For $A \in \Rnn$, $\mathcal{S}$ is said to be \emph{$A$-invariant}, $A \mathcal{S} \subset \mathcal{S}$, if $Ax \in \mathcal{S}$ for all $x \in \mathcal{S}$.
For a subset $\mathcal{S} \subset \mathds{Z}$, we write $g \geq 0$ or $g \in \mathds{R}^{\mathcal{S}}_{\geq 0}$ if  $g: 
	 \mathcal{S} \to \mathbb{R}_{\geq0}$ is a \emph{nonnegative function (sequence)} and
	\begin{align*}
	    \mathds{1}_{\mathcal{S}}(t) := 
	        \begin{cases}
	            1 & t \in \mathcal{S}\\
	            0 & t \notin \mathcal{S}
	        \end{cases}
	\end{align*}
	for the \emph{(1-0) indicator function}. In particular, we denote the \emph{Heaviside function} by $s(t) := \mathds{1}_{\mathbb{R}_{\geq 0}}(t)$ and the \emph{unit pulse function} by $\delta(t)$. 
The set of all \emph{absolutely summable sequences} is denoted by $\lint$ and the set of \emph{bounded sequences} by $\linf$.

\subsection{Linear discrete-time systems}
We consider \emph{discrete-time, linear time-invariant (LTI) systems} with input
$u$ and output $y$. The output $g(t) = y(t)$ corresponding to
$u(t) = \delta(t)$ is called the \emph{impulse response}. Throughout this
work, we assume that  $g \in \lint$ and $u \in \ell_\infty$. The \emph{transfer function} of the system is given by
\begin{equation}
G(z) = \sum_{t=0}^\infty g(t)z^{-t} = \frac{r \prod_{i=1}^{m} (z-z_i)}{\prod_{j=1}^{n}(z-p_i)}
\end{equation}
where $r \in \mathbb{R}$, and $p_i$ and $z_i$ are referred to as \emph{poles}
and \emph{zeros}, both of which are sorted in same way as the eigenvalues of
a matrix. Without loss of generality, we assume that $g(0) = 0$ ($m < n$).
The tuple $(A,b,c)$ is referred to as a \emph{state-space realization} of 
$G(z)$ if \cref{eq:SISO_d} holds, 
with stable $A \in \mathbb{R}^{n\times n}$, and $b, c^\transp \in \mathbb{R}^n$.
It holds then that 
\begin{equation*} 
    g(t) = c A^{t-1} b \, s(t-1).
\end{equation*}
We assume that the set of poles and the set of zeros of a transfer function are
disjoint, and define the \emph{order of a system} as the number of poles of
$G(z)$. A realization $(A,b,c)$ is called \emph{minimal} if the eigenvalues of
$A$ are precisely the poles of $G(z)$.  
For $t \ge 0$, the \emph{Hankel operator} 
	\begin{align}
	 \label{eq:def_hank_disc}
	(\Hank{g} u)(t) &:= \sum_{\tau=-\infty}^{-1} g(t-\tau)u(\tau) = \sum_{\tau=1}^{\infty} g(t+\tau)u(-\tau) %
	\end{align}
describes the evolution of $y$ after $u$ has been turned off at $t = 0$, i.e., 
$u(t) = u(t)(1-s(t))$. It obeys the factorization 
\begin{equation}
    \Hank{g} u = \mathcal{O}(A,c)(\mathcal{C}(A,b) u) \label{eq:Hankel_fac}
\end{equation}
with the \emph{controllability and observability operators} given by
\begin{subequations}
    \label{eq:statespace_operators}
	\begin{align}
	x(0) = \Con{}(A,b) u &:= \sum_{t=-\infty}^{-1} A^{-t-1} b u(t), \ u \in \linf  \\
	(\Obs{}(A,c) x_0)(t) &:= cA^t x_0, \ x_0 \in \mathds{R}^n,  \ t \in \mathds{Z}_{\geq 0}
	\end{align}
\end{subequations}
Finally, for $t,j \in \mathds{Z}_{>0}$, we will often make use of the \emph{Hankel matrix}   
\begin{subequations} 
    \begin{align}
	H_g(t,j) &:= \begin{pmatrix}
	g(t) & g(t+1) & \dots & g(t+j-1)\\
	g(t+1) & g(t+2) & \dots & g(t+j)\\
	\vdots & \vdots & \ddots   & \vdots \\
	g(t+j-1)   & g(t+j) & \dots & g(t-2(j-1))\\
	\end{pmatrix} = \Obs{j}(A,c) A^{t-1} \Con{j}(A,b) \label{eq:hankel_mat_decomp}
	\end{align}
	where 
	\begin{align}
	\label{eq:truncated_con}
	\Con{j}(A,b) &:= \begin{pmatrix}
	b & Ab & \dots & A^{j-1}b
\end{pmatrix}	\\
	\label{eq:truncated_obs}
	\Obs{j}(A,c) &:= \Con{j}(A^\transp,c^\transp)^\transp.
	\end{align}
\end{subequations}

\subsection{Total positivity and the variation diminishing property}
\label{sec:vardim}

A central idea in this work is that positivity is an instance of
the variation diminishing property. The \emph{variation} of a sequence or vector $u$ is defined as
the number of sign-changes in $u$, i.e., 
\begin{equation*}
\vari{u} := \sum_{i \geq 1} \mathbb{1}_{\mathbb{R}_{< 0}}(\tilde{u}_i \tilde{u}_{i+1}), 
    \quad 
    \vari{0} := 0
\end{equation*}
where $\tilde{u}_i$ is the vector resulting form deleting all zeros in $u$.
\begin{defn}\label{def:ovd_k}
A linear map $u \mapsto X u$ is said to be \emph{\vardimord{k}} ($\ovd{k}$), 
$k \in \mathds{Z}_{\ge0}$, if for all $u$ with $\vari{u} \leq k$ it holds that
\begin{enumerate}[i.]
    \item $\vari{Xu} \leq \vari{u}$.
    \item The sign of the first non-zero elements in $u$ and $Xu$ coincide whenever $\vari{u} = \vari{Xu}$.
\end{enumerate}
If the second item is dropped, then $u \mapsto Xu$ is called \emph{\vardim{k}} ($\vd{k}$). For brevity, we simply say $X$ is $\text{(O)}\vd{k}$.
\end{defn} 

The $\ovd{k}$ property extends the the cone-invariance of nonnegative matrices, 
namely $X \in \Rmn_{\geq 0}$ is $\ovd{0}$, because $X \mathds{R}_{\geq 0}^m 
\subseteq  \mathds{R}_{\geq 0}^n$. 
For generic $k$, \emph{total positivity theory} provides algebraic conditions for the $\ovd{k}$ property by means of compound matrices. To define these, let the $i$-th
elements of the $r$-tuples in
\begin{equation*}
\mathcal{I}_{n,r} := \{ v = \{v_1,\dots,v_r\} \subset \mathds{N}: 1\leq v_1 < v_2 < \dots < v_r \leq n \}
\end{equation*} 
be defined by \emph{lexicographic ordering}. Then, 
the $(i,j)$-th entry of the \emph{r-th multiplicative compound matrix} 
$\compound{X}{r} \in \mathbb{R}^{\binom{n}{r} \times \binom{m}{r}}$ of 
$X \in \Rmn$ is defined by $\det(X_{(I,J)})$, where $I$ is the $i$-th and $J$ is
the $j$-th element in $\mathcal{I}_{n,r}$ and $\mathcal{I}_{m,r}$, respectively. 
For example, if $X \in \mathbb{R}^{3 \times 3}$, then $\compound{X}{2}$ reads
\begin{align*}
\begin{pmatrix}
\det(X_{\{1,2 \},\{1,2 \}}) & \det(X_{\{1,2 \},\{1,3\}}) & \det(X_{\{1,2 \},\{2,3\}})\\
\det(X_{\{1,3 \},\{1,2 \}}) & \det(X_{\{1,3 \},\{1,3\}}) & \det(X_{\{1,3 \},\{2,3\}})\\
\det(X_{\{2,3 \},\{1,2 \}}) & \det(X_{\{2,3 \},\{1,3\}}) & \det(X_{\{2,3 \},\{2,3\}})\\
\end{pmatrix}.
\end{align*}
Notice a nonnegative matrix verifies $X_{[1]} = X \ge 0$, which is equivalent to $X$ being $\ovd{0}$. This can be generalized through the compound matrix as follows \cite{grussler2020variation,karlin1968total}.
\begin{defn}
    \label{def:k_pos_matrix}
	Let $X \in \Rmn$ and $k \leq \min\{m,n\}$. $X$ is called \emph{\tp{k}} if
	$\compound{X}{j} \ge 0$ for $1 \leq j \leq k$,  
	and \emph{\stp{k}} if $\compound{X}{j} > 0$ for 
	$1 \leq j \leq k$.
	In case $k = \min\{m,n\}$, $X$ is called \emph{(strictly) totally 
	positive}.
\end{defn}
\begin{prop} \label{prop:k_pos_mat_var}
	Let $X \in \Rmn$ with $n \geq m$. Then, $X$ is $\tp{k}$ with $1 \leq k \leq m$
	if and only if $X$ is $\ovd{k-1}$.
\end{prop}

The Cauchy-Binet formula implies the following important properties \cite{fiedler2008special}.
\begin{lem}\label{lem:compound_mat}
	Let $X \in \mathbb{R}^{n \times p}$ and $Y \in \mathbb{R}^{p \times m}$.
	\begin{enumerate}[i)]
		\item $\compound{(XY)}{r} = \compound{X}{r}\compound{Y}{r}$.
		\item $\sigma(\compound{X}{r}) = \{\prod_{i \in I} \lambda_i(X): I \in \mathcal{I}_{n,r} \}$.
		\item $\compound{X^\transp}{r} = (\compound{X}{r})^\transp$.
	\end{enumerate} 
\end{lem}

In conjunction with the
Perron-Frobenius theorem \cite{perron1907theorie,frobenius1912matrizen}, this yields a spectral characterization of $\tp{k}$ matrices as follows. 
\begin{cor}\label{cor:EW_compound}
	Let $X \in \Rnn$ be \tp{k} such that $\compound{X}{j}$ is irreducible for $1\leq j \leq k$. Then,
	\begin{enumerate}[i.]
		\item $\lambda_1(X) > \dots >\lambda_k(X) > 0$.
		\item $\lambda_1(\compound{X}{j}) = \prod_{i=1}^j \lambda_i(X) > 0$.
		\item $\compound{\begin{pmatrix}
		\xi_1 & \dots & \xi_j
		\end{pmatrix}}{j} \in 
		\mathds{R}^{\binom{n}{j}}_{>0}$, $1\leq j\leq k$,  where $\xi_i$ is the eigenvector associated with $\lambda_i(X)$ for $1 \leq i\leq k$. 
	\end{enumerate} 
\end{cor}

The next result shows that it often suffices to check consecutive minors to verify $k$-positivity vis-a-vis a combinatorial number of minors \cite{karlin1968total,fallat2017total}.
\begin{prop}
    \label{prop:consecutive}
	Let $X \in \Rmn$, $k \leq \min \{n,m\}$ be such that
	\begin{enumerate}[i.]
	    \item all consecutive $r$-minors of $X$ are 
	    positive, $1 \leq r \leq k-1$,
	    \item all consecutive $k$-minors of $X$ are
	    nonnegative (positive).
	\end{enumerate}
	Then, $X$ is (strictly) \tp{k}.
\end{prop}
Finally, to be able to apply \cref{prop:consecutive} to matrices lacking
strictly positive intermediate $j$-minors, we will make use of the
following.
\begin{prop}\label{prop:stp_aux_matrix}
Let $F(\sigma) \in \mathbb{R}^{n\times n}$ be given by 
$F(\sigma)_{ij}=e^{-\sigma(i-j)^2}$, with $\sigma \ge 0 $,
and let $X\in\mathbb{R}^{n\times m}$ with $m \le n$. Then for $r \le m$,
the following hold:
\begin{enumerate}[i.]
    \item $F(\sigma)$ is strictly totally positive.
    \item $F(\sigma)\to I$ as $\sigma \to \infty$, and 
    $F(\sigma)X \to X$ as $\sigma \to \infty$.
    \item if $\compound{X}{r}\ge 0$, and if $\rk\,X = m$, then
    $\compound{(F(\sigma)X)}{r}>0$ for all $\sigma > 0$.
    \item if $\compound{(F(\sigma)X)}{r}\geq 0$ for all $\sigma>0$, then $\compound{X}{r}\ge 0$.
\end{enumerate}
\end{prop}
\begin{proof}
Parts (i)-(iii) are proven in \cite[p.220]{karlin1968total}, while part (iv)
is a consequence of (ii) and the continuity of the minors of a matrix in its
entries (see~e.g.~\cite{horn2012matrix}).
\end{proof}

\subsection{Hankel $k$-positivity and compound systems}

The $\ovd{k}$ property of LTI systems \cref{eq:SISO_d} has been studied in \cite{grussler2020variation}, where a distinction is made between LTI systems with $\ovd{k}$ Toeplitz and Hankel operators. The latter are particularly relevant to this work.
\begin{defn}
\label{def:ex_Hankel}
    A system $G(z)$ is called \emph{Hankel $\tp{k}$} if $\Hank{g}$ is $\ovd{k-1}$
    ($k \geq 1$). If $k=\infty$, $G(z)$ is said to be \emph{Hankel totally positive}.  
\end{defn}
In other words, $G(z)$ is $\ovd{k-1}$ from past inputs to future outputs. Note that if $G(z)$ is Hankel $\tp{k}$, then it is also Hankel $\tp{j}$, $1\le j \le k$. Since 
an $\ovd{k-1}$ $\Hank{g}$ maps nonnegative inputs to nonnegative outputs, it can be verified that Hankel $1$-positivity coincides with the familiar property of {external positivity}.
\begin{defn}
    \label{def:external_positivity}
    $G(z)$ is \emph{externally positive} if 
    $y \in \mathds{R}^{\mathds{Z}_{\geq 0}}_{\geq 0}$ for all 
   $u \in \mathds{R}^{\mathds{Z}_{\geq 0}}_{\geq 0}$ (and $x(0) = 0)$.
\end{defn}
A central observation of \cite{grussler2020variation} 
is the following characterization involving $k$-positive matrices.
\begin{lem}\label{lem:var_dim_op}
    A system $G(z)$ is Hankel $k$-positive if and only if for all $j \in \mathds{Z}_{> 0}$, $H_g(1,j)$ is $\tp{k}$.
\end{lem}
Using \cref{prop:consecutive,prop:stp_aux_matrix}, it is easy to show that 
$k$-positivity of Hankel matrices only require checking the nonnegativity of 
consecutive minors \cite{fallat2017total}. 
From \cref{eq:hankel_mat_decomp}, each of these consecutive minors is given by
\[ 
    g_{[j]}(t) := \det(H_g(t,j)),
\] 
which is interpreted as the impulse response of an LTI system $G_{[j]}(z)$, called the \emph{$j$-th compound system}. The compound systems
feature in the following characterization.
\begin{prop} \label{prop:con_minor}
Given $G(z)$ and $1 \leq k \leq n$, the following are equivalent:
		\begin{enumerate}[i.]
		\item $G(z)$ is Hankel $\tp{k}$. 
		\item $G_{[j]}(z)$ is externally positive for $1 \leq j \leq k$. \label{item:Hankel_comp}
		\item $H_g(1,k-1) \succ 0$, $H_g(2,k-1) \succeq 0$ and $G_{[k]}(z)$ is externally positive. 
		\item $G_{[j]}$ is Hankel $\tp{k-j+1}$ for $1\leq j \leq k$.
	\end{enumerate} 
\end{prop}
In particular, the equivalence between Hankel $\ovd{0}$ and external positivity becomes evident as both properties require $g_{[1]} = g \geq 0$ \cite{Farina2000}.

A key fact for our new investigations is that 
if $(A,b,c)$ is a realization of $G(z)$, then $G_{[j]}(z)$ can be realized as
\begin{equation}
    \label{eq:comp_real}
    (\compound{A}{j}, \compound{\Con{j}(A,b)}{j},\compound{\Obs{j}(A,c)}{j}).
\end{equation} 
Note that by \cref{eq:hankel_mat_decomp}, $g_{[j]} = 0$ if $j >n$, which is why $k=n$ coincides with the case $k=\infty$. The following pole constraints of Hankel $\tp{k}$ systems will also be important for our new developments. 

\begin{prop}\label{prop:repeated}
	Let $G(z) = \sum_{a=1}^{l} \sum_{b=1}^{m_a} \frac{r_{ba}}{(z-p_a)^b}$ be Hankel $\tp{k}$. Then, $m_1 = \dots = m_{k-1} = 1$ and $p_{k-1} > 0$ if $k \leq \sum_{a = 1}^l m_a$. In particular, $G(z)$ is Hankel totally positive if and only if all poles are nonnegative and simple. 
\end{prop}

\section{Internally Hankel $k$-positive systems}
\label{sec:internal_hankel_kpos_DT}
In this section, we introduce and study a subclass of Hankel $\tp{k}$ systems 
which admit state-space realizations such that the $\ovd{k-1}$ property also holds
internally.
\begin{defn}\label{def:int_Hankel}
    $(A,b,c)$ is called \emph{internally Hankel $\tp{k}$} if $A$, $\Con{}(A,b)$, 
    and $\Obs{}(A,c)$ are $\ovd{k-1}$ ($k \geq 1$). If $k=\infty$, we
    say that $(A,b,c)$ is \emph{internally Hankel totally positive}.  
\end{defn}
\emph{Internally Hankel $\tp{k}$} systems are, therefore, $\ovd{k-1}$ 
from past input $u$ to $x(0)$, and from $x(0)$ to all future $x(t)$ and 
future output $y$. 
In particular, 
by \cref{eq:Hankel_fac}, 
all internally Hankel $\tp{k}$ systems are also Hankel $\tp{k}$, and
setting $u \equiv 0$ recovers the $\tp{k}$ property of autonomous
systems as partially studied in \cite{margaliot2018revisiting,alseidi2021discrete}. Thus, \cref{def:int_Hankel} bridges the external and the autonomous notions of
variation diminishing LTI systems.
In the remainder of this section, we aim to answer the following main questions:
\begin{enumerate}[I.]
    \item How does internal Hankel $k$-positivity manifest as tractable algebraic properties of $(A,b,c)$?
    \item When does a system have a minimal internally Hankel $k$-positive realization?
\end{enumerate}
Our answers will generalize the well-known case of $k = 1$
\cite{ohta1984reachability,anderson1996nonnegative,benvenuti2004tutorial,Farina2000,luenberger1979introduction}, which, we will see, coincides with the familiar class of internally positive systems \cite{Farina2000}.
\begin{defn}
    \label{def:internal_pos}
	$(A,b,c)$ is said to be \emph{internally positive}
	if for all $u \in \mathds{R}^{\mathds{Z}_{\geq 0}}_{\geq 0}$ and all $x(0) \ge 0$, it follows that $y \in \mathds{R}^{\mathds{Z}_{\geq 0}}_{\geq 0}$
	and	$x(t)\ge 0$ for all $t\ge 0$.
\end{defn}
In \Cref{sec:cont}, our findings are extended to continuous-time systems, 
and we use our result to establish a framework that upper bounds the variation of the impulse response in arbitrary LTI systems.

\subsection{Characterization of internally Hankel $k$-positive systems}
We start by recalling the following well-known characterization of internal positivity in terms of system matrix properties \cite{luenberger1979introduction}.
\begin{prop}
    \label{prop:internal_pos}
   $(A,b,c)$ is internally positive if and only if $A,b, c\geq 0$.
\end{prop}
Therefore, internal positivity indeed implies that $(A,b,c)$ is internally Hankel $1$-positive (through \cref{prop:k_pos_mat_var}). The converse can be seen from the following equivalences, which give a first characterization of internal Hankel $k$-positivity.  

\begin{lem}\label{lem:var_dim_ss}
For $(A,b,c)$, the following are equivalent:
\begin{enumerate}[i.]
    \item $\Con{}(A,b)$ and $\Obs{}(A,c)$ are $\ovd{k-1}$, respectively.
    \item For all $t \geq k$, $\Con{t}(A,b)$ and $\Obs{t}(A,c)$ are $\tp{k}$, respectively.
\end{enumerate}

In particular, $(A,b,c)$ is internally Hankel $k$-positive if and only if $A$, $\Con{t}(A,b)$ and $\Obs{t}(A,b)$ are $k$-positive for all $t \geq k$.
\end{lem}
\begin{proof}
	By \cref{prop:k_pos_mat_var}, it suffices to show that  $\Con{}(A,b)$ is $\ovd{k-1}$ if and only if  $\Con{t}(A,b)$ is $\ovd{k-1}$ for all $t \geq k$. For $\Obs{}(A,c)$, the proof is analogous via the duality \cref{eq:truncated_obs}.
	\\
	$\Rightarrow$: Follows by considering inputs
    $u$ with $u(t) = 0$ for $t \ge k$. 
    \\
    $\Leftarrow$: Let $u$ be an input with $\vari{u} \le k-1$. 
    Since 
    \[ \vari{\Con{t}(A,b) 
    \begin{pmatrix}
        u(t-1) & \dotsc & u (0)
    \end{pmatrix}^\transp} \le k-1 \] %
    for all $t > 0$, in the limit $t\to \infty$ we obtain $\vari{\Con{}(A,b) u} \le k-1$.
\end{proof}

Next, we want to find a finite-dimensional and, thus, certifiable characterization of 
internal Hankel $k$-positivity. To this end, we derive our first main result: a sufficient condition for $k$-positivity of the
controllablility and observability operators.

\begin{thm}
\label{thm:cont_k_pos} 
Let $(A,b,c)$ be a realization of $G(z)$ such that $A \in \mathds{R}^{K \times K}$ is $k$-positive. Then,   
\begin{enumerate}[i.]
    \setlength\itemsep{.25em}
	\item if $\compound{\Con{j}(A,b)}{j} \geq 0$ for $1\leq j \leq k$ and $\rk(A^{K-j}\Con{j}(A,b)) = j$ for all $1\leq j \leq k-1$, then
	$\Con{t}(A,b)$ is $k$-positive for all $t \geq k$.
	\item if $\compound{\Obs{j}(A,c)}{j} \geq 0$ for $1\leq j \leq k$ and $\rk(\Obs{j}(A,c)A^{K-j}) = j$ for all $1\leq j \leq k-1$, then $\Obs{t}(A,c)$ is $k$-positive for all $t \geq k$.
\end{enumerate}
The rank constraints are fulfilled if $k$ does not exceed the order of the system and $p_{k-1} > 0$. %
\end{thm}
\begin{proof}
Since the case $k=1$ is trivial, we assume $k>1$. We only prove the first item as the second follows by duality. Assume that (i) holds, and $F(\sigma)$ is as in \cref{prop:stp_aux_matrix}. We will show now by induction on $j$ that
for all $\sigma>0$ and $t\geq j$,
\begin{equation}\label{eq:induction_ass}
\compound{(F(\sigma)\Con{t}(A,b))}{j} \; \text{is  }
    \left\{
    \begin{array}{cc}
        \text{positive}, & \text{for } j<k  \\
        \text{nonnegative}, & \text{for } j=k 
    \end{array}
    \right.
\end{equation}
Then, by \cref{prop:stp_aux_matrix}, it follows that 
$\lim_{\sigma \to \infty}
\compound{(F(\sigma)\Con{t}(A,b))}{j}  = \compound{\Con{t}(A,b)}{j}$ is nonnegative for all 
$t \ge j$ and $1\leq j \leq k$, and thus $k$-positivity of $\Con{t}(A,b)$
is proven for all $t \ge k$.

To prove \eqref{eq:induction_ass}, first notice that if $\rk(A^{K-j}\Con{j}(A,b)) = j$, then through the Jordan form of $A$, it is easy to show that $\rk(A^{i}\Con{j}(A,b)) = j$ for all $i \in \mathds{Z}_{\geq 0}$. Further, $\compound{(A^i \Con{j}(A,b))}{j} = \compound{A}{j}^i \compound{\Con{j}(A,b)}{j}
\ge 0$  by
\cref{lem:compound_mat} and therefore \cref{prop:stp_aux_matrix} implies that for all
$\sigma>0$ and $i \in \mathds{Z}_{\geq 0}$:
\begin{equation}
    \label{eq:strictness_induction}
    \compound{(F(\sigma)A^i \Con{j}(A,b))}{j} \; \text{is }
    \left\{
    \begin{array}{cc}
        \text{positive}, & \text{for } j<k  \\
        \text{nonnegative}, & \text{for } j=k 
    \end{array}
    \right.
\end{equation}
We are now ready to prove the induction on $j$:

\emph{Base case} ($j=1$):
Taking $j=1$ in \cref{eq:strictness_induction}, it follows that
\begin{equation}
    \label{eq:induction_base}
    F(\sigma)\Con{t}(A,b)  \; \text{is 
    positive for all } \; t\geq 1.
\end{equation}

\emph{Induction step ($j>1$)}:
Let us now assume that \cref{eq:induction_ass} holds true for all 
$1 \leq j \leq j^\ast-1 < k$. We want to show that \cref{eq:induction_ass}
also holds for $j=j^\ast$. 
To this end, note that for any $t\ge j^\ast$, all consecutive $j^\ast$
columns of $F(\sigma) \Con{t}(A,b)$ are of the form 
$F(\sigma) A^i \Con{j^\ast}(A,b)$ 
for some 
$i\in\mathbb{Z}_{\ge 0}$. 
Thus by \cref{eq:strictness_induction} all consecutive $j^\ast$-minors of
$F(\sigma) \Con{t}(A,b)$ are 
positive (resp. nonnegative) when $j^\ast < k$ (resp. $j^\ast = k$).
This fact, in conjunction with the strict $(j^\ast-1)$-positivity of 
$F(\sigma) \Con{t}(A,b)$ (the induction hypothesis), implies through
\cref{prop:consecutive} that $F(\sigma) \Con{t}(A,b)$ is 
strictly $j^\ast$-positive when $j^\ast < k$, and 
$j^\ast$-positive when $j^\ast = k$. 
In particular, \cref{eq:induction_ass} holds for $j=j^\ast$, and the
induction is proven. 

Let us assume now that $(N,g,h)$ is an $n$-th order minimal realization of $(A,b,c)$, $k \leq n$ and
$p_{k-1} > 0$. By the Kalman controllability and observability forms there exists then a 
$T \in \mathds{R}^{K \times K}$ with
    \begin{align*}
        T^{-1} A T = 
        \begin{pmatrix} 
        N & 0 & \ast \\
        \ast & \ast & \ast\\
        0 & 0 & \ast
        \end{pmatrix}, \;  T^{-1}b = \begin{pmatrix} g \\ \ast \\ 0
    \end{pmatrix},\; c T \begin{pmatrix} h & 0 & \ast \end{pmatrix}
    \end{align*}
    Thus, $$A^i\Con{j}(A,b) = T \begin{pmatrix} N^{i}\Con{j}(N,g) \\ \ast \\ 0
    \end{pmatrix}$$ and $\rk(A^i\Con{j}(A,b)) \geq \rk(N^i\Con{j}(N,g))$. Then, since $p_{k-1} > 0$, it follows from the controllablity of $(N,g)$ that $\rk(N^{K-j}\Con{j}(N,g)) = j$ for all $1 \leq j \leq k-1$. Analogous considerations apply to the observability operator. This shows the claim on removing the rank constraint. 
\end{proof}
Combining \cref{prop:repeated,thm:cont_k_pos,lem:var_dim_ss} yields then the following
characterizations of internal Hankel $k$-positivity.
\begin{thm}\label{thm:Hankel_kpos_equiv}
    $(A,b,c)$ is internally Hankel $\tp{k}$ if and only if the realizations of the first $k$ compound systems of $(A,b,c)$ in \cref{eq:comp_real} are (simultaneously) internally positive. 
\end{thm}

\subsection{Internally Hankel $k$-positive realizations}
\label{sec:internally_kpos_discrete}
To approach the question of the existence of (minimal) Hankel $k$-positive realizations, we turn to an invariant cone approach, which has proven to be useful in dealing with the case $k=1$ \cite{ohta1984reachability,benvenuti2004tutorial}. The following is a classical
result.
\begin{prop}
	\label{prop:ext_int_pos} For $G(z)$, the following are equivalent:
	\begin{enumerate}
		\item $G(z)$ is externally positive with minimal realization $(A,b,c)$
		\item There exists an $A$-invariant proper convex cone $\mathcal{K}$ such that $b \in \mathcal{K}$ and $c^\transp \in \mathcal{K}^\ast$.
	\end{enumerate}
In particular, $G(z)$ has an internally positive realization if and only if $\mathcal{K}$ can be chosen to be polyhedral. 
\end{prop}
Several algorithms for finding such an invariant polyhedral cone can be found,
e.g., in \cite{farina1996existence,farina2011positive}. Internal positivity
is, therefore, a finite-dimensional means to verify external positivity. However, since not every externally positive system admits an internally positive 
realization \cite{ohta1984reachability,Farina2000}, we cannot expect 
that all externally positive compound systems have internally positive 
realizations, and, as a consequence, internal Hankel $k$-positivity does
not follow from its external counterpart. For Hankel \emph{total}
positivity, however, the two notions are equivalent.
\begin{prop}\label{prop:Hankel_total_real}
	$G(z)$ is Hankel totally positive if and only if  there exists a minimal realization that is internally Hankel totally positive. 
	\end{prop}
\begin{proof}
By \cref{prop:repeated}, it holds that $G(z) = \sum_{i=1}^{n}\frac{r_i}{z-p_i}$ with $p_i \geq 0$ and $r_i > 0$. This admits a realization $A = \diag(p_n,\dots,p_1)$ and $b = c^\transp$ with $b_i = \sqrt{r_{n-i+1}}$, $1 \leq i \leq n$. Thus, $A$ is totally positive, and by applying \cite[Lemma~22]{grussler2020variation} to the sub-matrices of $\Con{j}(A,b)$, also $\Con{j}(A,b) =  \Obs{j}(A,c)^\transp$ is totally positive for all $j \geq 1$. Thus, the result follows by \cref{thm:Hankel_kpos_equiv}. 
\end{proof}

To bridge the gap between external and internal Hankel $k$-positivity, we address the existence of minimal internal realizations. 
\begin{thm}
	\label{thm:real_simple}
	$G(z)$ with minimal realization $(A,b,c)$ has a minimal internally Hankel $k$-positive realization if and only if there exists a $P \in \mathds{R}^{n \times n}$ with $\rk(P) = n$ such that for all $1 \leq j \leq k$
        \begin{subequations}
		    \begin{align}
		    AP = PN, \text{$k$-positive $N$}  \label{eq:A_cond}\\
 		    \compound{\Con{j}(A,b)}{j} \in \cone(\compound{P}{j}), \label{eq:b_cond}\\
		    \compound{\Obs{j}(A,c)}{j}^\transp \in  \cone(\compound{P}{j})^\ast.     
		    \label{eq:c_cond}
		    \end{align}
	    \end{subequations}

	\end{thm}
\begin{proof} $\Rightarrow$: Let $(A_+,b_+,c_+)$ be a minimal internally Hankel internally $k$-positive realization. By the similarity of minimal realizations there exists an invertible $P \in \Rnn$ such that for $1 \leq j \leq k$
\begin{align*}
		AP = PN \text{ with $k$-positive $N = A_+$} \\
 		\compound{\Con{j}(A,b)}{j} = \compound{P}{j} \compound{\mathcal{C}^j(A_+,b_+)}{j}\\
		\compound{\Obs{j}(A,c)}{j} \compound{P}{j} = \compound{\mathcal{O}^j(A_+,c_+)}{j},
		\end{align*}
which by \cref{thm:Hankel_kpos_equiv} shows the claim.

$\Leftarrow$: If \cref{eq:A_cond,eq:b_cond,eq:c_cond} hold, then there exists a minimal internally positive realization $(N,g,h)$ with nonnegative 
\begin{align*}
    \compound{\Con{j}(N,g)}{j} &= \compound{P}{j}^{-1}\compound{\Con{j}(A,b)}{j}\\
    \compound{\Obs{j}(N,h)}{j} &= \compound{\Obs{j}(A,c)}{j} \compound{P}{j}
\end{align*}
for $1 \leq j \leq k$ and $k$-positive $N$. Thus, by \cref{thm:Hankel_kpos_equiv}, $(N,g,h)$ is internally Hankel $k$-positive.

\end{proof}

\begin{rem}\label{rem:thm_realization}
From \cref{prop:ext_int_pos}, we know that in case of $k=1$, \cref{thm:real_simple} remains true even if we drop minimality, i.e., $P \in \mathds{R}^{K \times K}$ with $K \geq n$. The reason for this lies in the fact that there always exists a controllable, internally positive $(A_+,b_+,c_+)$ \cite{Farina2000,ohta1984reachability}. To be able to conclude the same for $k >1$, we would need to show that \cref{eq:A_cond,eq:b_cond} are sufficient for the existence of $b_+$ with $b = Pb_+$ and $\compound{\Con{j}(A_+,b_+)}{j}\geq 0$ for $1\leq j \leq k$. Together with \cref{cor:impulse}, it is possible to show then that \cref{thm:real_simple} extends to non-minimal internally Hankel $k$-positive realizations, i.e., $P \in \mathds{R}^{n \times K}$ with $K > n$. 
\end{rem}
Finally, under an irreducibility condition, all autonomous $k$-positive systems give rise to an internally Hankel $k$-positive system.{
\begin{prop}\label{prop:irred_A_b}
	Let $A\in\mathbb{R}^{n \times n}$ be $k$-positive with irreducible $\compound{A}{j}$, $1\leq j \leq k$. Then there exists a $b \in \mathds{R}^n$ such that
	$\compound{\Con{j}(A,b)}{j} > 0$ for all $1\leq j \leq k$ and $(A,b)$ is controllable.
\end{prop}
	\begin{proof}
By \cref{cor:EW_compound}, $\lambda_1(A)>\dotsc> \lambda_k(A) > 0$. Let $\xi_1,\dots,\xi_k$ denote 
the associated eigenvectors. Our goal is to show that there exists $\alpha \in \mathds{R}^k$ with
$\alpha_1 \geq \dots \geq \alpha_k > 0$ such that $b = \sum_{j=1}^k \alpha_j \xi_j$ fulfills the first
part of the claim. Then, by continuity of the determinant there also exists such a $b$ with 
$(A,b)$ controllable. 
	
We begin by writing
	\[
		\Con{j}(A,b) = \begin{pmatrix}
		\alpha_1\xi_1 & \dots & \alpha_k \xi_k
		\end{pmatrix}
		V^j,
	\]
	where $V^j$ is the Vandermonde matrix
	\[
	V^j = 
	\left(
			\begin{array}{cccc}
				1 & \lambda_1(A) & \dotsc & \lambda_1(A)^{j-1} \\
				\vdots &\vdots & & \vdots \\
				1 & \lambda_k(A) & \dotsc & \lambda_k(A)^{j-1} 
			\end{array}
		\right),
	\]
	so that \cref{lem:compound_mat} implies
		\begin{align*}
	\label{eq:existence_b} 
	\compound{\Con{j}(A,b)}{j} &= \compound{\begin{pmatrix}
		\alpha_1\xi_1 & \dots & \alpha_k \xi_k
		\end{pmatrix}}{j} \compound{V^j}{j}.
	\end{align*}
	Since $\compound{V^j}{j}$ is a positive vector \cite[Example 0.1.4]{fallat2011}, we can absorb its contribution into the
	choice of $\alpha$ and assume without loss of generality that  
		\begin{align*}
	\compound{\Con{j}(A,b)}{j} = \compound{\begin{pmatrix}
		\xi_1 & \dots & \xi_k
		\end{pmatrix}}{j} \compound{\mathrm{diag}(\alpha_1,\dotsc,\alpha_k)}{j} e
	\end{align*}
	where $e$ is the vector of all ones. Thus, $\compound{\Con{j}(A,b)}{j}$ is a linear combination of the columns in $\compound{\begin{pmatrix}
		\xi_1 & \dots & \xi_k
		\end{pmatrix}}{j}$, where each column is multiplied by the diagonal entry in $\compound{\mathrm{diag}(\alpha_1,\dotsc,\alpha_k)}{j}$. In particular, the first column $\compound{\begin{pmatrix}
		\xi_1 & \dots & \xi_k
		\end{pmatrix}}{j}$ is 
	 positive by \cref{cor:EW_compound} and multiplied by the largest factor $\prod_{i=1}^{j} \alpha_i$. Therefore, by choosing inductively sufficiently large $\alpha_1 \geq \dots \geq \alpha_k > 0$, the entries in $\compound{\Con{j}(A,b)}{j}$ are dominated by $\prod_{i=1}^{j} \alpha_i \compound{\begin{pmatrix}
		\xi_1 & \dots & \xi_k
		\end{pmatrix}}{j}$, proving their 
positivity for $1\leq j \leq k$. 
\end{proof}}
An example why the irreducibility in \cref{prop:irred_A_b} cannot in general be dropped is given in \Cref{sec:example}. %

\section{Extensions}
\label{sec:cont}
In this section, we first discuss extensions of our results in discrete-time (DT) to continuous-time (CT)
systems, followed by applications to step-response analysis.

\subsection{Continuous-Time Systems}
The tuple $(A,b,c)$ is a CT state-space realization if 
\begin{equation}\label{eq:SISO_ct}
\begin{aligned}
\dot{x}(t) &= Ax(t) + bu(t), \\
y(t) &= cx(t), 
\end{aligned} 
\end{equation}
with $A \in \mathbb{R}^{n\times n}$, $b, c^\transp \in \mathbb{R}^n$. Its impulse response is
\begin{equation} 
    g(t) = c e^{A t} b s(t)
\end{equation}
and the controllability and observability operators are given by
\begin{subequations}
    \label{eq:CT_operators}
	\begin{align}
	\label{eq:CT_controllability}
	\Conct{}(A,b) u &:= \int_{-\infty}^{0_-} e^{-A\tau}b u(\tau) d\tau\\
	(\Obsct{}(A,c) x_0)(t) &:= ce^{A t} x_0, \ 
	x_0 \in \mathds{R}^n,  
	\ t\geq 0.
	\end{align}
\end{subequations}
As for DT systems, we assume $g$ to be absolutely integrable and $u$ to be bounded.  
By defining the variation of a continuous-time signal $u: \mathds{R} \to \mathds{R}$ as
\begin{equation*}
    \label{eq:CT_var}
    \varict{u} := \sup_{\substack{n \in \mathds{Z}_{>0} \\ t_1<\dotsc<t_n}} \; \vari{[u(t_1),\dotsc,u(t_n)]}
\end{equation*}
we can define CT internal Hankel $k$-positivity as follows.
\begin{defn}
    \label{def:CT_internal_hankel_ovd_k}
    A CT system $(A,b,c)$ is called \textit{CT internally Hankel $\tp{k}$} if $e^{At}$, $\Conct{}(A,b)$, and $\Obsct{}(A,c)$ are 
    $\ovd{k-1}$ for all $t \geq 0$.
\end{defn}
As in DT, we seek to characterize these systems through finite-dimensional $k$-positive constraints. We will do so by discretization
of \cref{eq:CT_operators}, which allows us to apply our DT results. To this end, consider for $h , j > 0$, 
the \emph{(Riemann sum) sampled controllability operator}
\begin{subequations}
\label{eq:discretized_ops}
\begin{equation}
    \label{eq:discretized_ctr}
    \begin{split}
    \Conct{j,h}(A,b) u &:= h 
    \sum_{i=-j}^{-1} e^{-A ih} b u(ih)  = 
    h e^{Ah} 
    \Con{j}(e^{Ah},b)
    \left(
    \begin{array}{c}
         u(-h)  \\
         u(-2h) \\
         \vdots \\
         u(-jh)
    \end{array}
    \right)
    \end{split}
\end{equation}
and the \emph{sampled observability operator}  
\begin{equation}
    \label{eq:discretized_obs}
    \begin{split}
    \left(
    \begin{array}{c}
         y(h)  \\
         y(2h) \\
         \vdots \\
         y(jh)
    \end{array}
    \right)
     &= \Obs{j}(e^{Ah},c) e^{Ah} x_0 =: \Obsct{j,h}(A,c) x_0.
    \end{split}
\end{equation}
\end{subequations}
Note that for each $u$ and $x_0$ with finite variation there exist sufficiently large $j$ and small $h$ such that $\varict{\Conct{}(A,b)u} = \vari{\Conct{j,h}(A,b) u}$ 
and $\varict{\Obsct{}(A,c) x_0} = \vari{\Obsct{j,h}(A,c) x_0}$. Thus, \cref{prop:k_pos_mat_var} allows connecting the $\ovd{k-1}$ property of the CT
operators \cref{eq:CT_operators} to $k$-positivity of the matrices $\Conct{j,h}(A,b)$ and 
$\Obsct{j,h}(A,c)$, where we consider all $j \ge k$. Since, by \cref{lem:compound_mat},
$k$-positivity of these matrices follows from $k$-positivity of $e^{Ah}$, $\Con{j}(e^{Ah},b)$, and $\Obs{j}(e^{Ah},c)$, we arrive at the following CT analogue of
\cref{lem:var_dim_ss}.
\begin{lem}
    \label{lem:internal_kpos_CT}
    A CT system $(A,b,c)$ is CT internally Hankel $\tp{k}$ if and only if there exists a $\varepsilon > 0$ such that the DT system $(e^{Ah}, b,c)$ is (DT) internally Hankel $\tp{k}$ for all $h \in (0,\varepsilon)$.
\end{lem}
Using \cref{thm:Hankel_kpos_equiv}, CT internal Hankel $k$-positivity can
be verified by checking that the realizations
\[
    (\compound{e^{Ah}}{j}, \compound{\Con{j}(e^{Ah},b)}{j},\compound{\Obs{j}(e^{Ah},c)}{j})
\]
are internally positive for $0\le j \le k$.
However, it is undesirable to do this for all sufficiently small $h$. Next, we will discuss how to eliminate this variable from the above characterization. A classical result in that direction states that $e^{Ah} \geq 0$ for all $h \in (0,\varepsilon)$, $\varepsilon > 0$, (and in fact all $h \geq 0$) if and
only if $A$ is Metzler (i.e., $A$ has nonnegative off-diagonal entries) \cite{Farina2000,berman1979nonnegative}. In general,
the compound matrix of $e^{At}$ can be expressed in terms of the \emph{additive compound matrix} 
\cite{muldowney1990compound}
\[ 
    A^{[j]} := \log(\compound{\exp(A)}{j}) = \left . \frac{d}{dh}\compound{e^{Ah}}{j} \right|_{h=0} 
\]  
which satisfies 
\begin{equation}
    \compound{e^{Ah}}{j} = e^{A^{[j]}h}. \label{eq:add_comp_exp}
\end{equation}
In other words, $e^{Ah}$ is $k$-positive for all $h \geq 0$ if and only if $A^{[j]}$ is Metzler for $1\leq j\leq k$. The next result will also allow us to remove $h$ from the conditions involving $\Con{j}(e^{Ah},b)$ and
$\Obs{j}(e^{Ah},c)$.
\begin{thm}\label{thm:CT_equiv}
Let $(A,b,c)$ be a CT system such that $A^{[j]}$ is Metzler for $1\leq j\leq k$. Then, the following holds:
\begin{enumerate}[i.]
    \item $\compound{\Con{j}(A,b)}{j} \geq 0$ for $1\leq j \leq k$ if and only if there exists a sufficiently small $\varepsilon > 0$ such that $\compound{\Con{j}(e^{Ah},b)}{j} \geq 0$ for all $1 \leq j \leq k$ and all $h \in (0,\varepsilon)$.
    \item $\compound{\Obs{j}(A,c)}{j} \geq 0$ for $1\leq j \leq k$ if and only if there exists a sufficiently small $\varepsilon > 0$ such that $\compound{\Obs{j}(e^{Ah},c)}{j} \geq 0$ for all $1 \leq j \leq k$ and all $h \in (0,\varepsilon)$.
\end{enumerate}
\end{thm}
\begin{proof}
We only show the first item, as the second follows analogously. Let us begin by showing that we can assume $\rk(\Con{j}(A,b)) = \rk(\Con{j}(e^{Ah},b))$. To see this, note that if $\rk(\Con{j}(A,b)) < j$, then all $A^{i}b \in \im{\Con{j-1}(A,b)}$ for all $i \geq j-1$, where $\im{\cdot}$ denotes the \emph{image (range)} of a matrix. In particular, $e^{Ah}b = \sum_{i=0}^{\infty} \frac{(Ah)^i}{i!}b \in \im{\Con{j-1}(A,b)}$ for all $h >0$ and thus, $\rk(\Con{j}(e^{Ah},b)) < j$. Conversely, if $\rk(\Con{j}(e^{Ah},b)) < j$, then by suitable column additions we also have $$\rk(\Con{j}(I-e^{Ah},b)\diag(1,h,\dots,h^{j-1})^{-1}) = \rk(\Con{j}(h^{-1}(I-e^{Ah}),b)) < j,$$ which, due to the lower semi-continuity of the rank \cite{hiriart2013variational}, proves in the limit $h \to 0$ that $\rk(\Con{j}(A,b)) < j$. Hence, we can assume that $\rk(\Con{j}(A,b)) = \rk(\Con{j}(e^{Ah},b)) = j$, as otherwise $\compound{\Con{j}(A,b)}{j}=0$ and the claim holds trivially.

Next, let $\compound{\Con{j}(A,b)}{j} \geq 0$ for $1\leq j \leq k$. Then, $\compound{(F(\sigma)\Con{j}(A,b))}{j} > 0$
for all $\sigma > 0$ by \cref{prop:stp_aux_matrix} (iii). Suitable column additions within $\Con{j}(A,b)$ yield
that this is equivalent to $\compound{(F(\sigma)\Con{j}(I+hA,b))}{j} > 0$ for all $\sigma, \; h > 0$. 
Since $(I+hA)b$ can be approximated arbitrarily well by $e^{Ah}b$ for sufficiently small $h > 0$, 
the continuity of the determinant \cite{horn2012matrix} implies the equivalence to $\compound{(F(\sigma) 
\Con{j}(e^{Ah},b))}{j} > 0$ for all $\sigma> 0$ and all sufficiently small $h > 0$. Hence, our claim follows for
for sufficiently small $h>0$ by invoking \cref{prop:stp_aux_matrix} (iv).
\end{proof}

Since $e^{Ah}$ has only positive eigenvalues, it follows from \cref{thm:cont_k_pos,thm:CT_equiv} 
that \cref{thm:Hankel_kpos_equiv} remains true in CT.
\begin{thm}\label{thm:Hankel_kpos_equiv_CT}
    $(A,b,c)$ is CT internally Hankel $\tp{k}$ if and only if 
            \begin{equation}
                \label{eq:CT_compound_real}
                (A^{[j]}, \compound{\Con{j}(A,b)}{j},\compound{\Obs{j}(A,c)}{j}).  
            \end{equation}
        is CT internally positive for all $1 \leq j \leq k$
\end{thm}
Note that our defined compound system realizations are indeed the CT compound systems, whose external positivity can be used to verify CT (external) Hankel $k$-positivity. 
Finally, an analogue to \cref{thm:real_simple} can be obtained by substituting \cref{eq:A_cond} with the condition stated in the following lemma, extending the corresponding result in \cite{ohta1984reachability}.
\begin{lem}
Let $A \in \Rnn$ and $P \in \mathds{R}^{n \times K}$. Then, $e^{At} P = Pe^{Nt}$ with $e^{Nt}$ $k$-positive for all $t\geq 0$ if and only if there exits a $\lambda > 0$ such that $(A+\lambda I_n) P = P (N+\lambda I_K)$ with $(N+\lambda I_K)^{[j]} \geq 0$ for all $1\leq j\leq k$. 
\end{lem}
\begin{proof}
We begin by remarking the following properties of the additive compound matrix \cite{muldowney1990compound}: let $X,Y \in \Rnn$ and $Z \in \mathds{R}^{n \times K}$
\begin{enumerate}
    \item $(X+Y)^{[j]} = X^{[j]} + Y^{[j]}$
    \item $\compound{(e^{Xt}Z)}{j} = e^{X^{[j]}t} \compound{Z}{j}$
\end{enumerate}

$\Leftarrow$: Let $(A+\lambda I_n) P = P (N+\lambda I_K)$ for some $\lambda > 0$ such that $(N+\lambda I_K)^{[j]} \geq 0$ for all $1\leq j\leq k$. Then, for all $\frac{i}{t} \geq \lambda$ with $t \geq 0$ and $i \in \mathds{N}$, it holds that $(A+\frac{i}{t}  I_n) P = P (N+\frac{i}{t}  I_K)$ and consequently
\begin{align*}
  e^{At} P = \lim_{i \to \infty}
  \left( I_n+\frac{At}{i}\right)^i P = P \lim_{i \to \infty}   \left( I_K+\frac{Nt}{i}\right)^i= P e^{Nt}.
\end{align*}
Moreover, by the first propety above $N^{[j]}$ is Metzler, which by \cref{eq:add_comp_exp} implies that $e^{Nt}$ is $k$-positive for $t \geq 0$.

$\Rightarrow$: Let $e^{At} P = Pe^{Nt}$ with $k$-positive $e^{Nt}$ for all $t \geq 0$. Then, by definition of the additive compound matrix and the properties above, it holds that
\begin{align*}
   (A+\lambda I_n)^{[j]}\compound{P}{j} &=  (A^{[j]}+\lambda I_n^{[j]}) \compound{P}{j} \\ 
   &=\compound{P}{j} (N^{[j]} + \lambda I_K^{[j]}) = \compound{P}{j} (N+ \lambda I_K)^{[j]}
\end{align*}
for all $\lambda \geq 0$,  $1\leq j \leq k$, where $N^{[j]}$ is Metzler by \cref{eq:add_comp_exp}. Thus, by choosing $\lambda$ sufficiently large, we conclude that $(N+ \lambda I_K)^{[j]}$ is nonnegative for all $1\leq j \leq k$. 
\end{proof}

\subsection{Impulse and step response analysis}
Next, we want to apply our results to the analysis of over- and undershooting in a
step response, a classical problem in control (see e.g. \cite{aastrom2010feedback}). For LTI systems, the total number of over- and undershoots equals the number of sign changes in the impulse response. While several lower bounds for these sign changes have been derived \cite{damm2014zero,swaroop1996some,elkhoury1993influence,elkhoury1993discrete}, fewer results seem to exist on upper bounds \cite{elkhoury1993discrete,elkhoury1993influence}.  

In our new framework, we observe that the impulse response of $(A,b,c)$ fulfils $g(t) = (\Obs{}(A,c)b)(t)$. Therefore, if $\Obs{}(A,c)$ is $\ovd{k-1}$, then the impulse response of $(A,b,c)$ changes its sign at 
most $\vari{b}$ times for all $\vari{b} \leq k-1$, and has the same sign-changing order as $b$ in case of an equal number of sign-changes. Similarly to \cref{thm:real_simple}, we conclude the following result.
\begin{thm} \label{cor:impulse}
Let $(A,b,c)$ be an observable realization of $G(z)$. Then, there exists a realization $(A_+,b_+,c_+)$ of $G(z)$ such that $A_+$ and $\Obs{}(A_+,c_+)$ are $\ovd{k-1}$ if and only if there exists a $P \in \mathds{R}^{n \times K}$, $K \geq n$ such that $AP = PA_+$ and $\compound{\Obs{j}(A,c)}{j}^\transp \in \cone(\compound{P}{j})^\ast$, $1 \leq j 
\leq k$. 
\end{thm}
\begin{proof}
$\Rightarrow$: Follows by \cref{lem:var_dim_ss} as in the proof to \cref{thm:real_simple}. 

$\Leftarrow$: Let $A$ and $\Obs{}(A_+,c_+)$ be $\ovd{k-1}$. Then by the Kalman observability decomposition, there exists a transformation $T = \begin{pmatrix} P \\ \ast  \end{pmatrix}$ such that
\begin{align*}
T A_+ = \begin{pmatrix} 
A & 0\\
\ast & \ast
\end{pmatrix} T, \quad c_+ = \begin{pmatrix} c & 0 \end{pmatrix} T.
\end{align*}
Consequently, $AP = PA_+$, $c_+ = cP$ and therefore, by \cref{lem:var_dim_ss},  $\compound{\Obs{j}(A,c)}{j} \compound{P}{j} = \compound{\Obs{j}(A_+,c_+)}{j}$.
\end{proof}

Since the realization $(A_+,b_+,c_+)$ may not be unique, it remains an open question how to minimize the
sign changes in $b_+$ in order to make the upper bound the least conservative. We leave an answer to this question for future work. It should be noted that the approach in \cite{elkhoury1993discrete} essentially corresponds to the case where a realization with a totally positive observability operator exits, because it assumes positive distinct real poles and real zeros, apart from multiple poles at zero. To simplify the treatment of multiple poles at zero in our framework, we remark the following corollary.
\begin{cor}\label{cor:poles_origin}
Let $(A,c)$ be such that $\compound{\Obs{j}(A,c)}{j} \geq 0$ for all $1\leq j \leq k$ and $\rk(\Obs{j}(A,c)) = j$ for all $1 \leq j\leq k-1$. Further, assume there exists an $\varepsilon > 0$ such that $A+\eta I$ is $k$-positive and $\rk(A+\eta I)$ is full for all $\eta \in (0,\varepsilon)$. Then, $\Obs{t}(A,c)$ is $\ovd{k-1}$.
\end{cor}
\begin{proof}
Since $\Obs{j}(A+\eta I,c)$ results from row additions in $\Obs{j}(A,c)$, it follows that 
$\compound{\Obs{j}(A,c)}{j} \geq 0$ if and only if $\compound{\Obs{j}(A+\eta I,c)}{j} \geq 0$, and 
$\rk(\Obs{j}(A,c)) = \rk(\Obs{j}(A+\eta I,c))$. Since, by assumption, $\rk(A+\eta I)$ is full, it 
suffices to check the rank condition on $\Obs{j}(A+\eta I,c)$ in \cref{thm:cont_k_pos} in order to conclude with \cref{lem:var_dim_ss} that $\Obs{}(A+\eta I,c)$ is $\ovd{k-1}$. By the continuity of the minors
\cite{horn2012matrix}, it follows then that also $\Obs{}(A,c)$ is $\ovd{k-1}$.
\end{proof}

\section{Examples}
\label{sec:example}

\subsection{Internal Hankel $k$-positivity}

Consider a system given by the realization
\[
    A_+ = 
    \begin{pmatrix}
    0.25  &  0.25  &  0.20 \\
    0.25  &  0.30  &  0.30 \\
    0.10  &  0.35  &  0.40
    \end{pmatrix},
    \,
    b_+ = c_+^\transp = 
    \begin{pmatrix}
    1 \\ 0.1 \\ 0
    \end{pmatrix}
\]
For this realization, we have
\[
    \Con{3}(A_+,b+) = 10^{-2}
    \begin{pmatrix}
   100  & 27.5  & 16.575 \\
   10 &  28 &  19.325 \\
    0  & 13.5 & 17.95
    \end{pmatrix},
\]
\[
    \Con{3}(A_+,b+)_{[2]} = 10^{-3}
    \begin{pmatrix} 
   252.5 &  176.675 &  6.73375 \\
   135 &  179.5 &  26.98625 \\
   13.5 &  17.95  & 24.17125 \\
    \end{pmatrix},   
\]
and $\Con{3}(A_+,b+)_{[3]} = 21.472625\cdot 
10^{-3}$. Furthermore, we have
\[
    A_{+[2]} = 10^{-2}
    \begin{pmatrix}
    1.25  &  2.5  &  1.5 \\
    6.25  &  8  &  3 \\
    5.75  &  7  &  1.5 
    \end{pmatrix}
\]
and $A_{+[3]} = \det A= -2.25 \cdot 10^{-3}$ 
(all numbers above are exact).
Several facts can be stated regarding this realization. Firstly, 
$\rk\,A = \rk\,\Con{3}(A_+,b_+) = 3$,
and thus the system is controllable. Furthermore, $A_+$
is $2$-positive, but not $3$-positive, while
$\Con{3}(A_+,b_+)$ is $3$-positive. It immediately
follows from \cref{thm:cont_k_pos} that the 
controllability operator $\Con{}(A_+,b_+)$ is
$2$-positive, which can readily be verified numerically.
Secondly, it can be verified (we omit the details) that
$\Obs{3}(A_+,b_+)$ is full-rank and $3$-positive; we
conclude from Theorems \ref{thm:Hankel_kpos_equiv} and 
\ref{thm:cont_k_pos} that the (minimal) realization
$(A_+,b_+,c_+)$ is internally Hankel $2$-positive, but
not $3$-positive (since $\compound{A_+}{3} =\det A < 0$).
The canonical controllable realization of $G(z)$ reads
\begin{equation}
   A = 
    \begin{pmatrix}
    0 & 1 & 0 \\
    0 & 0 & 1 \\
    -0.00225 & -0.1075 & 0.95
    \end{pmatrix},
    \,
    b = 
    \begin{pmatrix}
    0 \\ 0 \\ 1
    \end{pmatrix},
    \,
    c = 
    \begin{pmatrix}
    0.0058 \\ -0.6565 \\ 1.01
    \end{pmatrix}^\transp \label{eq:canon_ex1}
\end{equation}

which is not internally Hankel 
$k$-positive for any $k\ge 1$. For the two realizations above, the $P$ matrix from
\cref{thm:real_simple} is simply the canonical 
controllability state-transformation matrix, given by
\[ P = \Con{3}(A_+,b_+)\Con{3}(A,b)^{-1}. \]

To illustrate the variation-diminishing 
property, we show in Figures \ref{fig:impulse_2pos}-\ref{fig:impulse_contr} the time evolution of
$y_+(t) = (\Obs{}(A_+,b_+)x_0)(t)$ and 
$y(t)=(\Obs{}(A,b)x_0)(t)$ for
the initial condition $x_0 = (-40.5,0.9,0.015)^\transp$.
It can be seen that given $\vari{x_0}=1$, the internally
Hankel $2$-positive realization yields $\vari{y_+}=0$,
and the sign variation in $x_0$ is diminished; the
controllability canonical realization yields 
$\vari{y}=3$, and the variation in $x_0$ is increased.

\begin{figure}
    \centering
    \begin{tikzpicture}
		\begin{axis}[height=4.5cm,
			width=12cm,
			axis y line = left,
			axis x line = center,
			ylabel={$y_+(t)$}, 
			xlabel={$t$}, 
			xlabel style={right},
			xticklabel style={above,yshift=0.5ex},
 			xmax=15.9,
			]			
			\addplot+[ycomb,black,thick]
				table[x index=0,y index=1] 
				{./data/impulse_2pos.txt};	
				\end{axis}
    \end{tikzpicture}
    \caption{The output of the internally Hankel $2$-positive
    realization, $y_+(t) = (\Obs{}(A_+,b_+)x_0)(t)$, has a smaller variation than
    $x_0 = (-40.5,0.9,0.015)^\transp$
    \label{fig:impulse_2pos}.}
\end{figure}
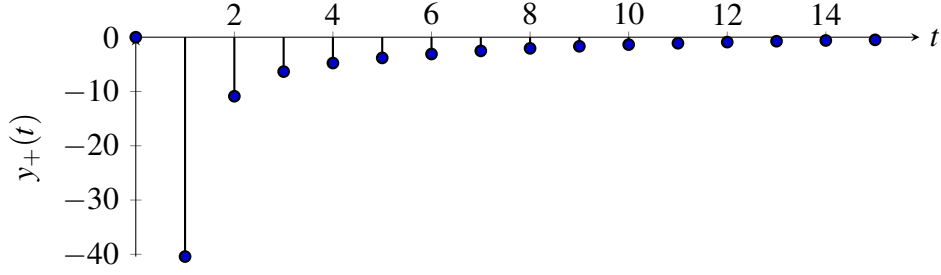

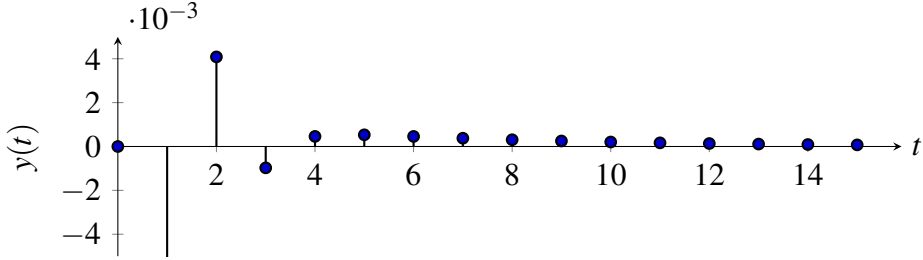
\begin{figure}
    \centering
    \begin{tikzpicture}
		\begin{axis}[height=4.5cm,
			width=12cm,
			axis y line = left,
			axis x line = center,
			ylabel={$y(t)$}, 
			xlabel={$t$}, 
			xlabel style={right},
 			xmax=15.9,
			ymax=0.005,
			ymin=-0.005,
			name = main,
			]			
			\addplot+[ycomb,black,thick]
				table[x index=0,y index=1] 
				{./data/impulse_contr.txt};	
		\end{axis}
    \end{tikzpicture}
    \caption{The output of the canonical controllable form
    realization \cref{eq:canon_ex1}, $y(t)=(\Obs{}(A,b)x_0)(t)$, has a larger variation than $x_0 = (-40.5,0.9,0.015)^\transp$.} %
    \label{fig:impulse_contr}
\end{figure}

\subsection{Impulse response analysis}

Consider the following system, previously shown as an example in
\cite{elkhoury1993discrete}:
\[
    G(z) = \frac{(z-0.22)(z-0.6)}{z^3(z-0.7)}
\]
The transfer function $G(z)$ has a realization given by
\begin{equation} \label{eq:real_impulse_example}
    A = 
    \begin{pmatrix}
        0.7 & 1 & 0 & 0 \\
        0 & 0 & 1 & 0 \\
        0 & 0 & 0 & 1 \\
        0 & 0 & 0 & 0 
    \end{pmatrix}
    \,
    b = 
    \begin{pmatrix}
    0 \\ 1 \\ -0.82 \\ 0.132 
    \end{pmatrix}
    \,
    c = 
    \begin{pmatrix}
    1 \\ 0 \\ 0 \\ 0 
    \end{pmatrix}^\transp 
\end{equation}

It can be verified that this realization has totally
positive $A$ and $\Obs{4}(A,c)$.
Furthermore, since $A$ is upper triangular with band-width $1$, $A + \eta I$ is totally
positive and $\rk(A+
\eta I)$ is full for any $\eta > 0$. Thus, by
\cref{cor:poles_origin}, the number of sign
changes in the impulse response of $G(z)$ (and, hence,
the number of extrema in its step response) is upper
bounded by $\vari{b}=2$; this same upper bound was
previously obtained by 
\cite{elkhoury1993discrete}. \cref{fig:impulse_ex2} shows that the
this bound is tight. 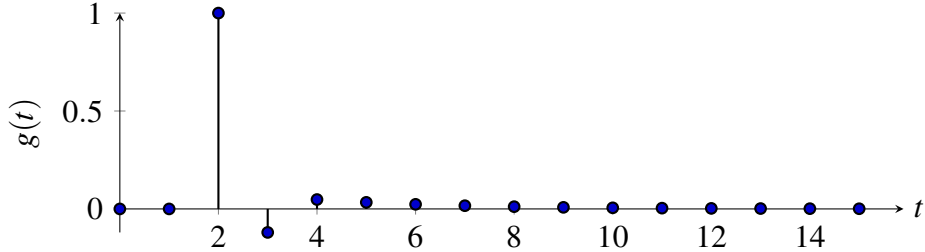
\begin{figure}
    \centering
    \begin{tikzpicture}
		\begin{axis}[height=4.5cm,
			width=12cm,
			axis y line = left,
			axis x line = center,
			ylabel={$g(t)$}, 
			xlabel={$t$}, 
			xlabel style={right},
 			xmax=15.9,
			]			
			\addplot+[ycomb,black,thick]
				table[x index=0,y index=1] 
				{./data/impulse_ex2.txt};	
				\end{axis}
    \end{tikzpicture}
    \caption{Impulse response of the system in \cref{eq:real_impulse_example}
    \label{fig:impulse_ex2} has two zero crossings, which coincides with our derived upper bound.}
\end{figure}
However, in contrast to \cite{elkhoury1993discrete}, our framework does not assume real poles or zeros. In particular, the modified transfer function
 $$G_m(z) = \frac{(z-0.5 + i)(z-0.5-i)}{z^3(z-0.7)},$$
can be realized with the same $A$ and $c$ as in \cref{eq:real_impulse_example} and with $b = \begin{pmatrix}
       0 & 1 & -1 & 1.25
\end{pmatrix}^\transp$, which again provides a tight upper bound on the variation of the impulse response. 

Finally, note that by \cref{prop:repeated}, there cannot be any $b$ such that $\Con{}(A,b)$ is $\tp{2}$, because otherwise $(A,b,c)$ would be Hankel $\tp{2}$. This illustrates the importance of the irreducibility condition in \cref{prop:irred_A_b}.

	\section{Conclusion}
Under the assumption of $k$-positive autonomous dynamics, this work has derived tractable conditions for which the controllability and observability operators are $k$-positive. These results have been used in two ways. 

First, we introduced and studied the notion of internally Hankel $k$-positive systems, i.e., systems which are variation diminishing from past inputs with at most $k-1$ variations to future states to future outputs. It has been shown that these properties are tractable  through internal positivity of the associated compound systems. In particular, internal Hankel $k$-positivity provides a means of studying external Hankel $k$-positivity with finite-dimensional tools. As a result, this systems class combines and extends two important system classes: (i) the celebrated class of internally positive systems ($k=1$) \cite{Farina2000} and (ii) the class of relaxation systems \cite{willems1976realization} ($k=\infty$); the latter has also been shown to admit minimal internally Hankel totally positive realizations. Moreover, our results lay the groundwork for future work linking autonomous variation diminishing systems, as considered in \cite{margaliot2018revisiting,alseidi2021discrete}, with the theory of externally variation diminishing systems \cite{grussler2020variation}. Finally, as a generalization of the case $k=1$ found in \cite{ohta1984reachability}, a characterization of when an externally Hankel $\tp{k}$ system possesses a minimal internally Hankel $\tp{k}$ realization has been discussed. In future work, the characterizations for non-minimal realizations and realization algorithms shall be addressed. Noticeably, we have not introduced an internal notion for externally Toeplitz $k$-positive systems: this is a consequence of the non-separability of the Toeplitz operator. Thus, contrary to the standard definition of internal positivity, this suggests that the Hankel operator is a more natural object with which to associate internal positivity.

Second, we have developed a new framework for upper-bounding the number of sign changes in the impulse response of an LTI system. In particular, while the results of  \cite{elkhoury1993discrete} are recovered in the case $k=\infty$, our framework allows considering generic $k$. In future work, we plan to address the conservatism of our analysis, its numerical numerical tractability, and the theoretical implication of the location of zeros. Further, we believe that a non-linear extension of our framework will be of timely importance. For instance, the cumulative difference between a step response and the output, called the (static) regret, is a common tractable measure in online learning \cite{orabona2019modern} and adaptive control problems \cite{karlsson2020feedback}. However, its meaningfulness depends on a small variability such as a small variance or a bounded variation.

\printbibliography

\end{document}